\documentclass[11pt]{amsart}
\usepackage{amsmath}
\usepackage{amsfonts}
\usepackage{amssymb}
\usepackage{amscd}
\usepackage[arrow,matrix,curve,cmtip,ps]{xy}
\usepackage{graphicx}

\usepackage{framed,color}
\definecolor{shadecolor}{rgb}{1,0.8,0.3}

\numberwithin{equation}{section}

\usepackage{amsthm}

\theoremstyle{plain}
\newtheorem{thm}{Theorem}[section] 
\newtheorem{theorem}[thm]{Theorem}
\newtheorem{proposition}[thm]{Proposition}
\newtheorem{corollary}[thm]{Corollary}
\newtheorem{lemma}[thm]{Lemma}

\theoremstyle{definition}

\newtheorem{question}[thm]{Question}
\newtheorem{remark}[thm]{Remark}

\theoremstyle{remark}

\newcommand{\cl}[1]{\mathcal{#1}}

\newcommand{\bb}[1]{\mathbb{#1}}

\begin{document}

\title{On a non-commutative analogue of a classical result of Namioka and Phelps}
\author{A. Samil Kavruk}
\begin{abstract}

A classical result of Namioka and Phelps states that the square 
is a test object to verify semi-simplexity in the tensor theory of convex compact sets. 
By using the quantization of generalized Namioka-Phelps test spaces we formulate a 
nuclearity criteria for C*-algebras, which establishes a non-commutative version of their result.
The proof we suggest covers the nuclearity characterization
via non-commutative polyhedron outlined by Effros \cite{Effros}. Several matrix systems studied by Farenick and Paulsen \cite{farenick--paulsen2011} are shown to be
test systems for nuclearity. We also prove that the standard Namioka-Phelps test space 
is C*-nuclear. We propose a 
partition of unity property for C*-algebras which distinguishes nuclear C*-algebras
among the others.


\end{abstract}

\maketitle
Non-commutative cubes and polygons have a subtle role in the study
of quantum correlations on a bipartite quantum
systems and are essential instruments of proving the equivalence of
Connes' embedding problem \cite{Connes} and the matricial Tsirelson's problem
in quantum information theory  \cite{FKPT-discrete}, \cite{Junge}, \cite{Horodecki}. Further properties of these objects, including
their representations via quotients and duality, are  extensively studied in \cite{FKPT-discrete} within
a more general context, namely, operator systems arisen from group representations. Similar objects used
in the study of injectivity in von Neumann algebras \cite{Effros}. The equivalence of 
weak expectation property and tight Riesz interpolation property, along with the Riesz decomposition property \cite{FKPT-discrete}, 
is based on nuclearity related properties of non-commutative cubes \cite{kavruk2012}.
In this paper we obtain further properties of these objects from a predual viewpoint and formulate a non-commutative
analogue of a classical result of Namioka and Phelps on function systems \cite{NP}. Our characterization also yields
a non-commutative partition of unity property that characterizes nuclear C*-algebras.

\vspace{0,2cm}

A C*-algebra $\cl A$ is called \textit{nuclear} if the injective and the projective tensor products coincide on $\cl A \otimes \cl B$ for every C*-algebra $\cl B$. 
Choi and Effros characterize nuclearity by completely positive factorization and approximation properties \cite{ChoiEffros0}, 
which complete the sufficiency shown in \cite{Lance2}. Injectivity of bidual von Neumann algebra, following its function 
theoretic predecessor (below), exhibits another formulation of nuclearity \cite{ChoiEffros}. The techniques
accumulated in this context have a fundamental role in general operator theory, in particular in nuclearity theory of operator
systems \cite{Han-Paulsen}, \cite{KirchbergCAR} e.g.
Our main goal in this paper is to obtain objects that verify nuclearity.
 
\vspace{0,2cm}

We let $\mathfrak{r}\mathfrak{K}$ denote the category of real Kadison spaces \cite{Kadison} and set 
$$
\cl R_{k,n} = \{ (a_i) \in \ell^\infty_{kn}(\bb R)  : \sum_{i=1}^k a_i =  \sum_{i=k+1}^{2k} a_i = \cdots =  \sum_{i=kn-k+1}^{nk} a_i  \}.
$$
A result of Namioka and Phelps states that  the state space of $\cl R_{2,2}$, the square, 
is a test object to verify semi-simplexity in the theory of tensor products of compact convex sets. 
They further prove that a Kadison space is nuclear if and only if its state space is semi-simplex \cite{NP}.
Perhaps their nuclearity characterization via Choquet simplexes can be best summarized if one manifests the nuclearity quartette in the category of Kadison spaces:

\begin{theorem}
The following are equivalent for $\cl V\in\mathfrak{rK}:$
\begin{enumerate}
 \item $\cl V$ is nuclear, that is, for every $\cl W \in \mathfrak{rK}$ we have a canonical order isomorphism
$\cl V\otimes_{\epsilon}\cl W = \cl V\otimes_{\pi}\cl W;$
 \item $\cl V^{**}$, the bidual Kadison space, is an injective object in $\mathfrak{rK}$ (consequently of the form $C(X)$, where $X$ is Stonean);
 \item the identity on $\cl V$ can be approximated via unital, 
positive maps through finite dimensional real $l^{\infty}$ spaces;
 \item we have a canonical order isomorphism $\cl V\otimes_{\epsilon}\cl R_{2,2} = \cl V\otimes_{\pi}\cl R_{2,2}$.
\end{enumerate}
\end{theorem}
\noindent Here $\otimes_{\epsilon}$ and $\otimes_{\pi}$ denote the injective and projective tensor products in $\mathfrak{rK}$. 
We refer the reader to \cite{Effros} for details.
The equivalence of (1) and (4) is due to Namioka and Phelps \cite{NP}. 
We refer $\cl R_{2,2}$ as the Namioka-Phelps test space and $\cl R_{n,k}$ as the 
generalized Namioka-Phelps spaces. In the above theorem $\cl R_{2,2}$ can be replaced by  
$\cl R_{n,k}$ for any $n\geq 2$, $k\geq 2$. 

$ $

By an Archimedean order unit space, AOU-space in short, we mean a complex Kadison space (or a complex function system) \cite{Paulsen-Tomforde}, \cite{Psulsen-Tomforde-Todorov}. 
The complexified versions of Namioka-Phelps test spaces are simply
$$
\cl W_{n,k} = \{(a_i)\in \ell^\infty_{kn}:
 \sum_{i=1}^k a_i =  \sum_{i=k+1}^{2k} a_i = \cdots =  \sum_{i=kn-k+1}^{nk} a_i\}.
$$
The statements of above theorem equally hold in AOU-space category. $\cl W_{n,k}$ has also a natural
operator system structure, which coincides with the super minimal quantization of $\cl W_{n,k}$ in the sense of \cite{Xabli}.
The duality correspondence between the test spaces and the non-commutative polygons is observed independently in \cite{FKPT-discrete} and \cite{OZAWA-duality}.
To be more precise, let $*_{i=1}^n\bb Z_k$ be the full free product of cyclic group $\bb Z_k$ and let C*($*_{i=1}^n\bb Z_k$)
denote its full group C*-algebra. We define the non-commutative polygon
$$
\cl{NP}_{n,k} = {\rm span} \{\lambda(g_i^j)\}_{j=0;i=1}^{k-1;n} \subset C^*( *_{i=1}^n\bb Z_k),
$$ 
where $g_i^j$ is the $j^{th}$ power of the generator $g_i$ of $\bb Z_k$ appearing on the $i^{th}$ order. 

\begin{theorem}[\cite{FKPT-discrete}, \cite{OZAWA-duality}] $\cl {NP}_{n,k}^* \cong \cl W_{n,k}$ completely order isomorphically. 
\end{theorem}

We shall start by proving:
\begin{theorem}\label{thm-main}
A unital C*-algebra $\cl A$ is nuclear if and only if we have a canonical complete order isomorphism
$$
\cl A \otimes_{\min} \cl W_{n,k} = \cl A \otimes_{\max} \cl W_{n,k}
$$
for every $n,k$ $($equivalently, for some $n\geq 2$, $k\geq 2$ with $(k,n)\neq (2,2)$.$)$
\end{theorem}

This characterization allows us to recover the equivalence of semi-nuclearity and nuclearity
in the sense of C. Lance \cite{Lance}. Moreover, we obtain that C$^*_u(\cl W_{n,k})$, the universal C*-algebra of
$\cl W_{n,k}$ \cite{KW}, is a C*-algebraic
object to verify nuclearity (see Corollary \ref{cor-main}). Also see Remark \ref{rem-polyhedra} for an observation on the
nuclearity characterization via Effros' non-commutative polyhedron.
An operator system $\cl S$ is said to be \textit{C*-nuclear} if its minimal and maximal
operator system tensor product with every C*-algebra coincides. 
In sharp contrast to the role of $\cl W_{2,2}$ in AOU-space category we have:
\begin{theorem}\label{thm-W22}
$\cl W_{2,2}$ is C*-nuclear.
\end{theorem}
\noindent The proof relies on the fact that the C*-nuclearity is 
preserved under duality which we sketch in the appendix. (This also settles the problem
exposed by the author in \cite{kavruk2011}.)

\cite{farenick--paulsen2011} exhibits several examples of operator system quotient and duality. In particular, we shall consider
the dual of \textit{words of length two}:
$$
\cl E_n = \{ [a_{ij}] \in M_n: a_{ii} = a_{jj} \mbox{ for all } i,j = 1,2,...,n  \}
$$
and the dual of the universal operator system generated by $n$-contraction:
$$
\cl U_n = \left\{ \bigoplus_{k=1}^n \left[\begin{array}{cc}  a_{11}^k & a_{12}^k \\   a_{21}^k & a_{22}^k  \end{array} \right]:\;
a_{ii}^l = a_{jj}^m \; \mbox{ for }1\leq l,m \leq n, \; 1 \leq i,j \leq 2 \right\}.
$$
The technique we develop allows us to extend  our nuclearity characterization:

\begin{theorem}\label{thm-Farenick-Paulsen Spaces}
Consider the following collection of operator systems:
$
\cl C = \{ \cl E_n\}_{n=3}^\infty \cup \{ \cl U_n\}_{n=2}^\infty.
$
A unital C*-algebra $\cl A$ is nuclear if and only if there exists $\cl S \in \cl C$ such that
we have a complete order isomorphism
$
\cl A \otimes_{\min} \cl S = \cl A \otimes_{\max} \cl S.
$
\end{theorem}



Generalized test spaces exhibit another formulation of Connes' embedding problem on the embedding of II$_1$-factors \cite{Connes}
in terms of approximate injectivity.
We let $\bb F_{\infty}$ denote the free group on countably many generators and $ \mbox{C*}(\bb F_{\infty})$ denotes the full group C*-algebra.
\begin{theorem}\label{thm Connes}
We fix $n\geq 2$, $k\geq 2$ with $(n,k)\neq (2,2)$. The following are equivalent:
\begin{enumerate}
 \item Connes' embedding problem has an affirmative answer;

\item for every ucp map $\varphi: \cl W_{n,k} \rightarrow \mbox{\rm C*}(\bb F_{\infty})$ and for every $\epsilon>0$ there exists a ucp map
$\tilde \varphi: \ell_{nk}^\infty \rightarrow \mbox{\rm C*}(\bb F_{\infty})$ such that $\|  \varphi - \tilde \varphi|_{\cl W_{n,k}}    \|_{cb} \leq \epsilon$;

\item every positive map $\varphi: \mbox{\rm C*}(\bb F_{\infty}) \rightarrow \cl W_{n,k}$ is nuclear.
\end{enumerate}
\end{theorem}

Let $\cl A$ be a unital C*-algebra. For self-adjoint elements $a$ and $b$ we shall write $a < b$ if there is a positive scalar $\delta$ such that
$b-a > \delta 1_{\cl A}$. We will say that $\cl A$ has \textit{m-partition of unity property} if for every
$b_1,b_2,...,b_m$ in $\cl A_{sa}$ with $$
0< b_1,b_2,...,b_m < 1
$$
there is a positive integer $n$, a positive element $[a_{ij}]$ in $M_n(\cl A)$ with
$
a_{11} + a_{22} + \cdots + a_{nn} = 1
$
and matrices $C_k = [c^k_{ij}]$, $k=1,2,...,m$, in $M_n$ with $0 < C_1,C_2,...,C_m < I_n$ such that
$$
b_k = \sum_{i,j = 1}^n c_{ij}^k a_{ij}, \;\; \mbox{ for } k = 1,2,...,m.
$$
(The term $k$ appear in matrix entries of $C_k$ are upper indexes.) It follows that:
\begin{theorem}\label{POUP}
\begin{enumerate}
 \item Every unital C*-algebra has 2-partition of unity property.

\item A unital C*-algebra is nuclear if and only if for every $n$, $M_n(\cl A)$ has $m$-partition of unity property for every $m$ $($equivalently for $m=3)$. 
\end{enumerate}
\end{theorem}

If $\cl A$ is a commutative C*-algebra then a careful analysis of our proof indicates that, for fixed
$0 < f_1,...,f_m <1$, the matrices [$a_{ij}$] and $C_k = [c_{ij}^{k}]$ appears in the 
definition can be chosen diagonal, therefore the above definition captures a well-known partition of unity property for continuous functions on
a compact set based on Arzel\`{a}-Ascoli theorem.

$ $

Preliminary section briefly summarizes the basics in operators systems such as duality, quotients and tensor products.
We also recall some fundamental facts regarding the nuclearity related properties, in particular,
Choi, Effros and Lance's nuclearity characterization and Kirchberg's WEP characterization. 
The second section is devoted for the proofs of the main results. Third section will include the partition of unity property for C*-algebras.
In appendix we will prove that C*-nuclearity is preserved under duality.


\section{Preliminaries}

In this section we establish the terminology and state basic definitions and results that
shall be used throughout the paper. An operator system can be defined concretely as a unital $*$-closed subspace of $B(H)$, bounded linear
transformations acting on a Hilbert space $H$. We refer the reader to \cite{PaulsenBook} for
the abstract characterization of these objects via compatible, strict collection of matricial cones
along with an Archimedean matrix order unit \cite{ChoiEffros77}. A map between operator systems
$ \varphi: \cl S \rightarrow \cl T$ is called completely positive, cp in short, if the $n^{\rm th}$-amplification
$id_n \otimes \varphi: M_n \otimes \cl S \rightarrow M_n \otimes \cl T$ is positive for all $n$.
If $\varphi$ is also unital, i.e. $\varphi(1_{\cl S}) = 1_{\cl T}$, we will say that $\varphi$ is a ucp map.
$\mathfrak{S}_n(\cl S) = \{ \varphi: \cl S \rightarrow M_n: \; \varphi \mbox{ is ucp } \}$ denotes the $n^{\rm th}$ matricial state space of $\cl S$.

\subsection{Duality} The Banach dual $\cl S^*$ of an operator system $\cl S$ can be endowed with a matricial
order structure via matricial cp maps. More precisely, after defining the self-adjoint idempotent $*$ via $f^*(s) = \overline{f(s^*)}$, we declare
$$
(f_{ij}) \in M_n(\cl S^*) \mbox{ positive if } \cl S \ni s \longmapsto (f_{ij}(s)) \in M_n \mbox{ is cp}. 
$$
The collection of the cones of the positive elements $\{M_n(\cl S^*)^+\}_{n=1}^\infty$ forms a strict, compatible matricial order structure on $\cl S^*$.
In general an Archimedean matrix order may fail to exist for this matricially ordered space. If dim$(\cl S) < \infty$ then a
faithful state $w$ on $\cl S$ can be assigned an Archimedean order unit for $\cl S^*$ \cite{ChoiEffros77}.


\subsection{Quotients} A subspace $J \subset \cl S$ is called a kernel if $J$
is kernel of a ucp map defined from $\cl S$ (equivalently kernel of a cp map).
A kernel is typically a non-unital $*$-closed subspace but these properties, in general, do not characterise a kernel.
A matricial order structure on the algebraic quotient $\cl S/J$ can be defined
by
$$
Q_n = \{ (s_{ij} + J): (s_{ij}) \in M_n(\cl S)^+) \}.
$$
The Archimedeanization process, in other words, completion of the
cones $\{Q_n\}$ relative to order topology induced by $(e +J) \otimes I_n$ (see \cite{Paulsen-Tomforde}, \cite{KPTT Tensor}),
yields the operator system quotient $\cl S/J$.
The universal property of the quotient ensures that if $\varphi: \cl S \rightarrow \cl T$
is a ucp map then the the induced map $\dot{\varphi}: \cl S/{\rm ker}(\varphi) \rightarrow \cl T$ is again a ucp map \cite{farenick--paulsen2011}.
$\varphi$ is called a \textit{quotient}  (resp.,\textit{ complete quotient}) map if $\dot{\varphi}$ is an order (resp. a complete order)  inclusion.
In particular a surjective ucp map $\varphi$ is completely quotient if and only if the adjoint $\varphi^\dag: \cl T^* \rightarrow \cl S^*$ is a complete order inclusion.



\subsection{Minimal tensor product} For operator systems $\cl S$ and $\cl T$ we define
$$
C_n^{\min} = \{ [x_{ij}] \in M_n(\cl S \otimes \cl T): \; [(\phi \otimes \psi)(x_{ij})] \geq 0 \;
\forall \phi \in \mathfrak{S}_p(\cl S),\; \psi \in \mathfrak{S}_q(\cl T),\; \forall\; p,q  \}.
$$
The collection of cones $\{C_n^{\min}\}_{n=1}^\infty$ forms a strict compatible matricial ordering for the algebraic tensor $\cl S \otimes \cl T$.  
Moreover, $1_{\cl S} \otimes 1_{\cl T}$ is a  Archimedean matricial order unit. Therefore the triplet
$
(\cl S \otimes \cl T, \{C_n^{\min}\}_{n=1}^\infty, 1_{\cl S} \otimes 1_{\cl T})
$
forms an operator system which we call the minimal tensor product of $\cl S$ and $\cl T$ and denote by $\cl S\otimes_{\min} \cl T$. 
We refer the reader to \cite{KPTT Tensor} for details. The minimal tensor product is spatial, injective and functorial. By the representation
of the minimal tensor we mean
CP$(\cl S, \cl T) \cong (\cl S^*\otimes_{\min} \cl T)^+
$
for any operator systems $\cl S$ and $\cl T$ with dim$(\cl S)<\infty$ \cite{kavruk2011}.

\subsection{Maximal tensor product} Let $\cl S$ and $\cl T$ be two operator systems. We define
$$
D_n^{\max} = \{ X^* (S \otimes T) X:  \; S \in M_p(\cl S)^+, \; T \in M_q(\cl T)^+,\; X \mbox{ is }pq\times n \mbox{ matrix}, \; p,q \in \bb N \}.
$$
The collection of the cones $\{D_n^{\max}\}_{n=1}^\infty$ are strict and compatible. Moreover, $1\otimes 1$ is a matricial order unit
for the matrix ordered space $(\cl S \otimes \cl T, \{D_n^{\max}\})$. Nonetheless $1\otimes 1$ may fail to be Archimedean, which can be resolved by
Archimedeanization process. We define
$$
C_n^{\max} = \{  X \in M_n(\cl S \otimes \cl T): X + \epsilon (1\otimes 1)_n \in D_n^{\max} \mbox{ for all } \epsilon >0    \}.
$$
The collection $\{C_n^{\max} \}$ forms a strict, compatible matrix ordering on $\cl S\otimes \cl T$
for which $1\otimes 1$ is an Archimedean matrix order unit. We let $\cl S\otimes_{\max} \cl T$
denote the resulting tensor product. max is functorial and projective \cite{KPTT Tensor}, \cite{Han}.
By the representation of the maximal tensor product we mean the canonical identification
CP$  (\cl S \otimes_{\max} \cl T, \bb C ) \cong$CP$(\cl S, \cl T^*)$.

\subsection{Some nuclearity related results} The nuclearity criteria of a 
C*-algebra $\cl A$ via injectivity of the bidual
von Neumann algebra $\cl A^{**}$ is essential for our work \cite{ChoiEffros}.
We also need the injectivity characterization of a von Neumann
algebra $\cl B \subset B(H)$ given by continuity of
$$
\cl B \otimes_{\min} \cl B' \rightarrow B(H), \;\; b \otimes b' \rightarrow bb',
$$ 
where $\cl B'$ denotes the commutant of $\cl B$ \cite{EffrosLance}. 
A C*-algebra $\cl A$ is said to have the \textit{weak expectation property} (WEP) (or weak injectivity) if the canonical inclusion of $A$ into $\cl A^{**}$ decomposes through
an injective object via ucp maps. We will need the following:
\begin{theorem}[Kirchberg, \cite{Kirchberg94}] 
A unital C*-algebra $\cl A$ has WEP if and only if
$$
\cl A \otimes_{\min}  C^*(\bb F_{\infty}) = \cl A \otimes_{\max}  C^*(\bb F_{\infty}).
$$
\end{theorem}
\noindent Here $\bb F_{\infty}$ can be replaced by $G_1*G_2$ for any discrete groups with cardinality $|G_1| \geq 2$, $|G_2| \geq 2$ with $|G_1| +|G_2| \geq 5$ \cite{Harpe}. Following duality characterization of
the minimal and maximal tensor products will be needed:
\begin{theorem}[Farenick, Paulsen] For finite dimensional operator systems $\cl S$ and $\cl T$ we have the complete order 
isomorphisms
$
(\cl S \otimes_{\min} \cl T)^* = \cl S^* \otimes_{\max} \cl T^* $ and $
(\cl S \otimes_{\max} \cl T)^* = \cl S^* \otimes_{\min} \cl T^*.
$
\end{theorem}

\section{Main Results}

Consider an abstract operator system $\cl S$ and a unital C*-algebra $\cl A$ such that $\cl S$ is an $\cl A$-bimodule. 
We let $\cdot$ denote the action of $\cl A$ on $\cl S$ 
which satisfies
$ a \cdot 1_{\cl S} = 1_{\cl S} \cdot a \mbox{ for all } a \in \cl A.
$
If in addition the action is compatible with the matricial positivity structure of $\cl S$, that is,
$$
[a_{ij}] \cdot [s_{ij}] \cdot [a_{ij}]^* \in M_{n}(\cl S)^+ \mbox{ for every } [s_{ij}] \in M_n(\cl S)^+ \mbox{ and } [a_{ij}] \in M_n(\cl A),
$$
we shall call $\cl S$ an operator $\cl A$-system (see \cite{PaulsenBook}). A ucp map $\varphi: \cl S \rightarrow \cl T$, where $\cl S$ and $\cl T$ 
are operator $\cl A$-systems, is called an $\cl A$-bimodule map 
if $\varphi(a\cdot s) = a \cdot \varphi(s)$ and $\varphi(s\cdot a) = \varphi(s)\cdot a$ for all $a\in \cl A$ and $s \in \cl S$. Finally,
an operator subsystem $S_0$ of an operator $\cl A$-system $\cl S$ is called an \textit{operator $\cl A$-subsystem} if $a\cdot s_0 \in \cl S_0$
for every $s_0 \in \cl S_0$.

\begin{proposition}
Let $\cl S,$ $\cl T$ be operator $\cl A$-systems and $\cl S_0$ be a operator $\cl A$-subsystem of $\cl S$. If
$\varphi: \cl S_0 \rightarrow \cl T$ is a ucp $\cl A$-bimodule map then every ucp
extension $\tilde \varphi: \cl S \rightarrow \cl T$ of $\varphi$ is also a $\cl A$-bimodule map.
\end{proposition}

\begin{proof}
Let $s$ be an element in $ \cl S$ with $\|s\| \leq 1$. It is not hard to show that
$$
\left[\begin{array}{ccc}
1 & 1 & s \\
1 & 1 & s \\
s^* & s^* & 1 
\end{array}\right] \geq 0 \mbox{ in } M_3(\cl S).
$$
In fact, if one considers $\cl S$ as a concrete operator subsystem of a C*-algebra
then the above element can be written as $[1\; 1 \; s]^* [1 \; 1\; s] + [0 \; 0 \; (1 - s^*s)^{1/2}]^*[0 \; 0 \; (1 - s^*s)^{1/2}]. $ 
(In the following matrices we will drop the module action notation $\cdot$ for simplicity.) Thus,
$$
\left[\begin{array}{ccc}
1 & 0 & 0 \\
0 & a^* & 0 \\
0 & 0 & 1 
\end{array}\right]
\cdot
\left[\begin{array}{ccc}
1 & 1 & s \\
1 & 1 & s \\
s^* & s^* & 1 
\end{array}\right] \cdot
\left[\begin{array}{ccc}
1 & 0 & 0 \\
0 & a & 0 \\
0 & 0 & 1 
\end{array}\right]
=
\left[\begin{array}{ccc}
1 & a & s \\
a^* & a^*a & a^*s \\
s^* & s^*a & 1
\end{array}\right] \geq 0.
$$
Since $\tilde \varphi$ is 3-positive, we have
$$
\left[\begin{array}{ccc}
\tilde \varphi(1) & \tilde \varphi(a) & \tilde \varphi(s) \\
\tilde \varphi(a^*) & \tilde \varphi(a^*a) & \tilde \varphi(a^*s) \\
\tilde \varphi(s^*) & \tilde \varphi(s^*a) & \tilde \varphi(1)
\end{array}\right] = 
\left[\begin{array}{ccc}
1 & a & \tilde \varphi(s) \\
a^* & a^*a & \tilde \varphi(a^*s) \\
\tilde \varphi(s^*) & \tilde \varphi(s^*a) & 1
\end{array}\right] \geq 0.
$$
This means that the following multiplication
$$
\left[\begin{array}{rrr}
a^* & -1 & 0 \\
0 & 0 & -1
\end{array}\right] \cdot
\left[\begin{array}{ccc}
1 & a & \tilde \varphi(s) \\
a^* & a^*a & \tilde \varphi(a^*s) \\
\tilde \varphi(s^*) & \tilde \varphi(s^*a) & 1
\end{array}\right] \cdot
\left[\begin{array}{rr}
a & 0 \\
-1 & 0 \\
0 & -1 
\end{array}\right]
$$
which is equal to
$$
\left[\begin{array}{cc}
0 & \tilde \varphi(a^*s)-a^*\tilde \varphi(s) \\
\tilde \varphi(s^*a)-\tilde \varphi(s^*)a & 1 
\end{array}\right]
$$
must be positive. So $\varphi(a^* \cdot s) = a^* \cdot \tilde \varphi(s)$ and $\tilde \varphi(s^* \cdot a)=\tilde \varphi(s^*) \cdot a$, which proves our claim.
\end{proof}

Our special interest on the theory of operator $\cl A$-systems arises from the following:
if $\cl S$ is an operator system and $\cl A$ is a C*-algebra then $\cl A \otimes_{\max} \cl S$
is an operator $\cl A$-system with the action given by
$
a \cdot (b\otimes s) = (ab) \otimes s
$
and $ (b\otimes s) \cdot a= (ba) \otimes s$.
The reader may refer to the proof of \cite[Theorem 6.7]{KPTT Tensor} for this fact.


\begin{proof}[proof of Theorem \ref{thm-main}]
Since a nuclear C*-algebra is in particular a (min,max)-nuclear operator system we only work on the non-trivial direction.
Fix $n\geq 2, \; k\geq 2$ with $(n,k)\neq (2,2)$. Let $\pi: \cl A \rightarrow B(H)$ be a unital $*$-homomorphism.
We may consider $B(H)$ as an operator $\cl A$-system with the action $a\cdot T = \pi(a)T$.
Note that our assumption in the theorem is equivalent to unital complete order inclusion
$$
\cl A \otimes_{\max} \cl W_{n,k} \subset \cl A \otimes \ell^{\infty}_{nk}.
$$
We may view $\cl A \otimes_{\max} \cl W_{n,k}$ as an operator $\cl A$-subsystem of  $\cl A \otimes \ell^{\infty}_{nk}$.
Let $\varphi: \cl W_{n,k} \rightarrow \pi(\cl A)'$ be a ucp map, where $'$ denotes the commutant. We wish to show that
$\varphi$ extends to a ucp map from $\ell^\infty_{nk}$ to $\pi(\cl A)'$.
Notice that the map $\pi \otimes \varphi : \cl A \otimes_{\max} \cl W_{n,k} \rightarrow  B(H)$ given by $a\otimes s \mapsto \pi(a) \varphi(s)$
is a ucp map which is also an $\cl A$-bimodule map.
Let $\tilde \varphi : \cl A \otimes \ell^{\infty}_{nk} \rightarrow B(H)$ be a ucp extension of $\varphi$.
By the previous proposition $\tilde \varphi$ is a also a $\cl A$-bimodule map. This means that
$$
\pi(a)\tilde \varphi(1 \otimes s) = a \cdot \tilde \varphi(1 \otimes s) = \tilde \varphi( a \cdot  (1 \otimes s) ) =
\tilde \varphi( a \otimes s) = \tilde \varphi( (1 \otimes s) \cdot a ) = \tilde \varphi(1 \otimes s) \pi(a)
$$
for any $a\in \cl A$ and $s \in \ell_{nk}^{\infty}$. If we define $\tilde \varphi_0 : \ell_{nk}^{\infty} \rightarrow B(H)$
by $\tilde \varphi_0 (s) = \tilde \varphi (1_{\cl A} \otimes s)$ then the image of $\tilde \varphi_0$ is contained
in $\pi(\cl A)'$. Thus we have proven that every ucp map $\varphi: \cl W_{n,k} \rightarrow \pi(\cl A)'$
extends to ucp map on $\ell_{nk}^{\infty}$. Since $\pi(\cl A)'$ is a von Neumann algebra
it follows that every completely positive (cp) map from $\cl W_{n,k}$ to $\pi(\cl A)'$  extends to cp
map on  $\ell_{nk}^{\infty}$ \cite{Effros-Ruan Book}. Let $Q: (\ell_{nk}^{\infty})^* \rightarrow \cl W_{n,k}^* $
be the completely quotient proximinal map obtained by taking 
the adjoint of the inclusion $\cl W_{n,k} \subset  \ell^{\infty}_{nk}.$
By the representation of minimal tensor product we have that
$$
(\ell_{nk}^{\infty})^* \otimes \pi(\cl A)' \xrightarrow{ Q \otimes id }\cl W_{n,k}^* \otimes_{\min} \pi(\cl A)'
$$
is a quotient map. In fact, a positive element $x$ in $\cl W_{n,k}^* \otimes_{\min} \pi(\cl A)'$
corresponds to a cp map $\varphi_x: \cl W_{n,k} \rightarrow \pi(\cl A)' $. Then its cp extension $\tilde \varphi_x$ on
$\ell^{\infty}_{nk}$ yields the desired positive representation of $x$ in $(\ell_{nk}^{\infty})^* \otimes \pi(\cl A)'$. Also note that,
by the projectivity of the maximal tensor product
$$
(\ell_{nk}^{\infty})^* \otimes \pi(\cl A)' \xrightarrow{ Q \otimes id }\cl W_{n,k}^* \otimes_{\max} \pi(\cl A)'
$$
is also a completely quotient map. (Here we are using the fact that $(\ell_{nk}^{\infty})^* \cong \ell_{nk}^{\infty}$
is nuclear.) Thus we obtain
$$
\cl W_{n,k}^* \otimes_{\min} \pi(\cl A)'  = \cl W_{n,k}^* \otimes_{\max} \pi(\cl A)'
$$
\textit{order} isomorphically. A moment of thought shows that these two operator systems must be completely
order isomorphic. In fact,
$$
M_n(\pi(\cl A)') \cong (\;(\pi \oplus \cdots \oplus \pi) (A)\; )'
$$
where the direct sum includes $n$ copies of $\pi$ and is defined from $\cl A$ into $M_n(B(H))$ diagonally. Since order isomorphism holds
for a generic $\pi$, it must hold for $\pi \oplus \cdots \oplus \pi$ as well, which implies that the above order isomorphism holds completely.

Let $\rho: \cl A \rightarrow B(H)$ be the universal representation
of $\cl A$, that is, $\rho$ is the direct sum of all cyclic representations of $\cl A$ so that we have $\cl A^{**} \cong \rho(\cl A)'{ }'$
weak*-ultraweak homeomorphically \cite{pisier_intr}.
By using the identity
$
\cl W_{n,k}^* \cong \cl NP_{n,k}
$
we have that $$\cl {NP}_{n,k} \otimes_{\min} \rho(\cl A)' = \cl NP_{n,k} \otimes_{\max} \rho(\cl A)'.$$
By the WEP criteria induced by $\cl{NP}_{n,k}$ \cite{FKTT Char.WEP}, we have that $\rho(\cl A)'$ must have WEP, or equivalently, it must be injective.
A classical result of Effros and Lance  \cite[Prop. 3.7]{EffrosLance} implies that the commutant $\rho(\cl A)''$ of $\rho(\cl A)'$
is also injective. Since we identify
$\cl A^{**}$ and $\rho(\cl A)''$, a well-known result of Choi and Effros \cite{ChoiEffros}
implies that $\cl A$ is nuclear.
\end{proof}

\begin{question}
If $\cl A \otimes_{\min} \cl  W_{n,k} = \cl A \otimes_{\max} \cl W_{n,k}$ ``order'' isomorphically for all $n$ and $k$ can we conclude
that $\cl A$ is nuclear?
\end{question}

\begin{corollary} \label{cor-main}
$\rm{(1)}$ If $\cl I$ is an ideal of a nuclear unital C*-algebra $\cl A$ then $\cl A / I$ is also nuclear.

$\rm{(2)}$ If $\cl A \otimes_{\max} \cl B \subset \cl A \otimes_{\max} \cl C$ for all C*-algebras $\cl B \subset \cl C$, then $\cl A$ is nuclear.

$\rm{(3)}$ A unital C*-algebra $\cl A$ is nuclear if and only if $\cl A \otimes_{\min} C_u^*(\cl W_{3,2}) =  \cl A \otimes_{\max} C^*_u(\cl W_{3,2})$.
\end{corollary}

\begin{proof}
(1) For any C*-algebra $\cl A$ and ideal $I \subset \cl A$ we have
$$
\frac{\cl A \otimes_{\min}\cl W_{n,k}}{\cl I \otimes_{}\cl W_{n,k}} = \frac{\cl A}{\cl I} \otimes_{\min} \cl W_{n,k} \mbox{ and }
\frac{\cl A \otimes_{\max}\cl W_{n,k}}{\cl I \otimes_{}\cl W_{n,k}} = \frac{\cl A}{\cl I} \otimes_{\max} \cl W_{n,k}.
$$
Here the first equality is direct consequence of the fact that $\cl W_{n,k}$ is exact \cite{kavruk--paulsen--todorov--tomforde2010}. The later
equality follows from the projectivity of the maximal tensor product \cite{Han}.
Thus, if $\cl A \otimes_{\min}\cl W_{n,k} = \cl A \otimes_{\max}\cl W_{n,k}$ then this is the case for $\cl A / \cl I$. The result follows from
Theorem \ref{thm-main}.

$ $

(2) In \cite{kavruk--paulsen--todorov--tomforde2010} it is shown that for any unital C*-algebra $\cl A$ and
operator system $\cl S$ we have a complete order inclusion
$$
\cl A\otimes_{\max} \cl S \subset \cl A\otimes_{\max} C^*_u(\cl S), 
$$
where $C^*_u (\cl S)$ denotes the universal C*-algebra of $\cl S$. Now, the condition in the theorem in particular implies
that
$$
\cl A \otimes_{\max} C^*_{u}(\cl W_{n,k}) \subset \cl A \otimes_{\max} C^*_{u}(\ell_{nk}^{\infty}).
$$
Combining this with the above result in \cite{kavruk--paulsen--todorov--tomforde2010}, we must have that
$$
\cl A \otimes_{\max} \cl W_{n,k} \subset \cl A \otimes_{\max} \ell_{nk}^{\infty}.
$$
Equivalently, min and max coincide on $\cl A\otimes \cl W_{n,k}$. So the result follows from
Theorem \ref{thm-main}.

$ $

(3) We only prove the non-trivial direction. As in proof of (2), we have
$$
\cl A \otimes_{\min} \cl W_{3,2} \subset  \cl A \otimes_{\min} C^*_u(\cl W_{3,2}) \;\;\mbox{ and }\;\; \cl A \otimes_{\max} \cl W_{3,2} \subset  \cl A \otimes_{\max} C^*_u(\cl W_{3,2}) 
$$
for every C*-algebra $\cl A$. Now assuming the condition in theorem we have that min and max coincide on $\cl A \otimes \cl W_{3,2}$
completely order isomorphically. By Theorem \ref{thm-main} $\cl A$ is nuclear.
\end{proof}

\begin{remark}\label{rem-polyhedra}
Following \cite{Effros} we define the non-commutative polyhedra
$$
\cl F_n = \left\{ [a_{ij}] \in M_{2n}: \sum_{i=1}^n a_{ii} = \sum_{i=n+1}^{2n} a_{ii}  \right\}.
$$
It is strightforward to see that $\cl W_{n,2}$ is embeddable into $\cl F_n$ with a conditional
expectation onto its range. In particular, if the minimal and the maximal
tensor products coincide on $\cl A \otimes \cl F_n$ then this is the case for
$\cl A \otimes \cl W_{n,2}$. Therefore we conclude that a unital C*-algebra $\cl A$ 
is nuclear if and only if there exists $n \geq 3$ such that we have a complete order ismorphism
$$
\cl A \otimes_{\min} \cl F_n = \cl A \otimes_{\max} \cl F_n,
$$
which covers the nuclearity characterization given in \cite{Effros}.
\end{remark}

\begin{proof}[proof of Theorem \ref{thm-W22}]
In Appendix we prove that C*-nuclearity is preserved under duality. Since $\cl {NP}_{2,2}$ is C*-nuclear \cite{FKPT-discrete}, its dual, namely $\cl W_{2,2}$,
is also C*-nuclear.
\end{proof}

We let $\cl U(\cl A)$ denote the unitary group of $\cl A$. The following is Corollary 5.8 of \cite{kavruk2011}:
\begin{proposition}
Let $\cl A$ and $\cl B$ be unital C*-algebras, $S \subset \cl A$ be an operator subsystem such that $C^*(S \cap \cl U(\cl A)) = \cl A$.
If $\cl B \otimes_{\min} \cl S = \cl B \otimes_{\max} \cl S $ then $\cl B \otimes_{\min} \cl A = \cl B \otimes_{\max} \cl A $.
\end{proposition}

\begin{proof}[proof of Theorem \ref{thm-Farenick-Paulsen Spaces}]
As a first step we will sharpen the property $\mathfrak{M}$ of \cite{farenick--paulsen2011}
by proving the following: if $n \geq 3$ then a canonical  complete order isomorphism of $\cl A \otimes_{\min} (M_n/J_n) $ and $\cl A \otimes_{\max} (M_n/J_n)$
implies that $\cl A$ has WEP. This is based on the representation of $M_n/J_n$ given in  \cite{farenick--paulsen2011}, in fact, the enveloping
C*-algebra  $C_e^*(M_n/J_n)$ is C*-isomorphic to $ C^*(\bb F_{n-1})$, moreover, $M_n/J_n$ can be identified, say by $i$, with a unitary operator subsystem
of $C^*(\bb F_{n-1})$ in the sense that the unitaries in $i(M_n/J_n)$ generates $ C^*(\bb F_{n-1})$ as a C*-algebra.
Now by Corollary 5.8 of \cite{kavruk2011}, the above proposition, we have that $\cl A \otimes_{\min} C^*(\bb F_{n-1}) = \cl A \otimes_{\max} C^*(\bb F_{n-1})$,
thus $\cl A$ has WEP (following from the Kirchberg's WEP characterization \cite{Kirchberg94}).

Turning back to proof of Theorem \ref{thm-Farenick-Paulsen Spaces}, if we replace the pair $\cl W_{n,k} \subset \ell_{nk}^{\infty}$
by $\cl E_n \subset M_n$ in the proof of Theorem \ref{thm-main}, since $\cl E_n^* \cong M_n / J_n$ by \cite{farenick--paulsen2011} and $M_n / J_n$
characterizes WEP as explained in the first paragraph, the proof equally holds. Similarly, if one replaces $\cl W_{n,k} \subset \ell_{nk}^{\infty}$ by $\cl U_n \subset \oplus_{i = 1}^n M_2$,
since $\cl U_n^* \cong \cl S_n$, the universal operator system generated by $n$-contraction \cite{farenick--paulsen2011},
and $\cl S_n$ characterizes WEP \cite{kavruk2011}, the proof similarly holds.
\end{proof}

In the rest of this section we will work on the formulations of Connes' embedding problem.
Recall that a linear map $\varphi: \cl S \rightarrow \cl T$, where $\cl S$ and $\cl T$ are operator systems, is called \textit{nuclear} 
if there is a net of completely positive maps $\phi_{\alpha}: \cl S \rightarrow M_{n(\alpha)}$, and $\phi_{\alpha}:  M_{n(\alpha)} \rightarrow \cl T$
such that for every $s \in \cl S$ we have $\psi_{\alpha} \circ \phi_{\alpha}(s) \rightarrow \varphi(s)$. (Here the maps $\{\phi_\alpha\}$ and $\phi_{\alpha}$ can also be taken ucp, we leave the verification to the reader.)

\begin{lemma}\label{lem-CEP1}
Let $\cl S$ and $\cl T$ be operator systems with dim$\cl (T) < \infty$. A linear map
$ f : \cl S \otimes \cl T \rightarrow \bb C$ is positive with respect to minimal tensor product
if and only if the associated map $\varphi_f: \cl S \rightarrow \cl T^* $, given by
$\varphi_f(s) (t) = f(s\otimes t)$ for $s \in \cl S$ and $t \in \cl T$, is nuclear.
\end{lemma}

\begin{proof} We first assume that $\varphi_f$ is nuclear and we will show that $f: \cl S \otimes_{\min} \cl T \rightarrow \bb C$
is positive. Let $\{\phi_\alpha, \phi_{\alpha}, M_{n(\alpha)}\}$ be the net appearing in the nuclearity definition of 
 $\varphi_f$. First note that
$$
\cl S \otimes_{\min} \cl T \xrightarrow{\;\;  \varphi_f \otimes id      \;\;} \cl T^* \otimes_{\max} \cl T
$$
is completely positive, in fact, $\varphi_f \otimes id$ can be approximated by the net of the composition
$$
\cl S \otimes_{\min} \cl T \xrightarrow{\;\;  \phi_\alpha \otimes id      \;\;} M_{n(\alpha)} \otimes_{\min} \cl T =
M_{n(\alpha)} \otimes_{\max} \cl T \xrightarrow{\;\;  \psi_\alpha \otimes id      \;\;} \cl T^* \otimes_{\max} \cl T.
$$
Let $E: \cl T^* \otimes_{\max} \cl T \rightarrow \bb C$ be the evaluation map, that is, $E(g\otimes t) = g(t)$. Note that the evaluation
map is positive, in fact based on the representation of maximal tensor product, $E$ corresponds to the identity $id: \cl T \rightarrow \cl T$, which is
a completely positive map. Now it is elementary to show that the following composition
$$
\cl S \otimes_{\min} \cl T \xrightarrow{\;\;  \varphi_f \otimes id      \;\;} \cl T^* \otimes_{\max} \cl T \xrightarrow{\;\; E\;\;} \bb C,
$$
which is positive, coincides with $f$. This proves one direction.

Conversely suppose that $f$ is continuous with respect to minimal tensor product. Let $\{\cl S_{\alpha}\}_{\alpha}$ be the collection of all finite dimensional operator
subsystem of $\cl S$ directed under inclusion. We let $f_\alpha: \cl S_{\alpha} \otimes_{\min} \cl T \rightarrow \bb C$ be the restriction of $f$, which is still positive. 
By using the duality of min and max \cite{farenick--paulsen2011}, $f_\alpha$ may be regarded as an element of $\cl S_\alpha^* \otimes_{\max} \cl T^*$. Now by the definition of maximal 
tensor product (and a compactness argument discussed in \cite{farenick--paulsen2011}),  there are integers $p$ and $q$, positive elements $[f_{ij}] \in M_p (\cl S_\alpha^*)$, $[g_{ij}] \in M_q (\cl T^*)$,
and a vector $x \in \bb C^p \otimes \bb C^q$ such that
$$
f_\alpha = x^* \left([f_{ij}] \otimes [g_{ij}]\right) x.
$$
We set $n(\alpha) = p$ and define $\phi_\alpha: \cl S_\alpha \rightarrow M_p$ by $\phi_\alpha (s) = [f_{ij}(s)]$, and $\psi_{\alpha}: M_p \rightarrow \cl T^*$
by $\psi_\alpha (A) = x^* (A\otimes [g_{ij}]) x$. Clearly $\phi_{\alpha}$ and $\psi_{\alpha}$
are completely positive maps with $\psi_{\alpha} \circ \phi_{\alpha} = f_\alpha$. Let
$\tilde \phi_{\alpha}$ be a completely positive extension of $\phi_\alpha: \cl S_\alpha \rightarrow M_{n(\alpha)}$. We claim that
$f$ can be approximated by $\psi_{\alpha} \circ \tilde \phi_{\alpha}$ in point norm topology. In fact, given $s\in \cl S$,
fix $\alpha_0$ so that $\cl S_{\alpha_0} \supset$span$\{1,s,s^*\}$. Note that for any $\alpha \geq \alpha_0$ we have
$\varphi_f(s) = \psi_{\alpha} \circ \tilde \phi_{\alpha} (s) $. This proves our claim.
\end{proof}

\begin{lemma}\label{lem-CEP2}
Let $\cl S_1 \subset \cl S_2$ be finite dimensional operator systems and $\cl A$ be a unital C*-algebra. We let
$Q$ denote the canonical quotient map $\cl S_2^* \rightarrow \cl S_1^*$. The following are equivalent:
\begin{enumerate}
 \item for every ucp map $\varphi: \cl S_1 \rightarrow \cl A $ and for every $\epsilon >0$, there exists
a ucp map $\psi :  \cl S_2 \rightarrow \cl A $ such that $\|\varphi - \psi|_{\cl S_1}\|_{cb} \leq \epsilon$;
 \item 
$
\cl S_2^* \otimes_{\min} \cl A \xrightarrow{\;\; Q\otimes id} \cl S_1^* \otimes_{\min} \cl A
$
is an order quotient map.
\end{enumerate}
\end{lemma}

\begin{proof} We first remark that if one replace the terms ucp by completely positive (cp) in (1) we obtain an equivalent statement.
In fact, assuming that every ucp map has an approximate ucp extension, one can construct an approximate extension
for a cp map $\varphi: \cl S_1 \rightarrow \cl A$ as follows: we first fix a faithful state $w$ on $\cl S_1$
and set  $ \varphi_{\delta} = \varphi + \delta w(\cdot)1_{\cl A}$ for some small positive scalar $\delta$.
By pre and post multiplying $\varphi_{\delta}$ by $\varphi_{\delta} (1)^{-1/2}$ we obtain
a ucp map. Now, by using an approximate ucp extension of $\varphi_{\delta} (1)^{-1/2} \varphi_{\delta} (\cdot) \varphi_{\delta} (1)^{-1/2}$,
and the fact that $\|\cdot \|_{cb}$ is a norm, one can easily construct an approximate cp extension for $\varphi$.
We leave the details to the reader. Similarly, assume that every cp map on $\cl S_1$ has an aproximate cp extension on $\cl S_2$. 
Then given a ucp map $\varphi: \cl S_1 \rightarrow \cl A$, it is possible to pick the approximate cp extension unital.
We similarly leave the elementary, but tedious, verification of this fact to the reader.

(1) $\Rightarrow$ (2): As a first step we fix a faithful state $w \in \cl S_1^*$ which we
consider as an Archimedean order unit. Let $\{ \delta_1,...,\delta_n \}$
be a basis for $\cl S_1^*$ formed by positive linear functionals. By rescaling, if necessary,
we may suppose that $\delta_i \leq w$, so $\|\delta_1\| \leq 1$ in $\cl S_1^*$. We will let $\{s_1,...,s_n\}$ be the pre-dual basis for $\cl S_1$
given by $\delta_i(s_j) = \delta_{ij}$. We fix a positive element $\Sigma \delta_i \otimes a_i$ in $\cl S_1^* \otimes_{\min} \cl A$.
Let $\varphi: \cl S_1 \rightarrow \cl A$
be the corresponding cp map (based on the representation of the minimal tensor product \cite{kavruk2011}).
By our assumption, for every $\epsilon>0$, we have a cp map $\psi: \cl S_2 \rightarrow \cl A$ 
with $\|\varphi - \psi|_{\cl S_1}\|_{cb} \leq \epsilon. $ Note that in particular we have
$\|\varphi(s_i) -  \psi(s_i)  \| \leq \epsilon$ where $\varphi(s_i) = a_i$ for each $i = 1,...,n$. Let $X \in \cl S_2^* \otimes_{\min} \cl A $
be the positive element corresponding $\psi$. We claim that
$\| \Sigma \delta_i \otimes a_i -(Q \otimes id) (X)    \| \leq \epsilon\, {\rm dim}(\cl S_1)$. In fact, it is
elementary to show that
\begin{eqnarray*}
\|  \sum_{i=1}^n (\delta_i \otimes a_i) - (Q \otimes id) (X) \|   &  =  & \|  \sum_{i=1}^n \delta_i \otimes a_i - \sum_{i=1}^n \delta_i \otimes \psi(s_i)  \|  \\
								  &  =  & \| \sum_{i=1}^n  \delta_i \otimes (\psi(s_i) - \varphi(s_i))   \|  \\
								  &  \leq  &   \sum_{i=1}^n \| \delta_i    
\| \| (\psi(s_i) - \varphi(s_i))   \| \leq \epsilon \,{\rm dim}(\cl S_1).
\end{eqnarray*}
Now by the Archimedeanization process described in \cite{kavruk--paulsen--todorov--tomforde2010}, $Q \otimes id$ is an order quotient map.

$ $

(2) $\Rightarrow$ (1): Let $\varphi: \cl S_1 \rightarrow \cl A$ be a cp map and let $u \in \cl S_1^* \otimes_{\min} \cl A$
be the corresponding positive element. Since $Q \otimes id$ is an order quotient map, for every $\epsilon>0$,
we have a positive element $U \in \cl S_2^* \otimes_{\min} \cl A$ such that $(Q\otimes id) (U) = u + \epsilon\, w \otimes 1 $,
where $w$ is a fixed faithful state on $\cl S_1$.
Let $\psi :\cl S_2 \rightarrow \cl A$ be the cp map corresponding $U$.  Note that
$\varphi - \psi|_{\cl S_1}  = \epsilon \, w (\cdot ) 1_{\cl A}  $. This clearly proves that (1) holds.
\end{proof}

\begin{proof}[proof of Theorem \ref{thm Connes}] Before we proceed we recall that Connes'
embedding problem (CEP) has an affirmative solution if and only if
we have a canonical complete order isomorphism
$$
 \cl {NP}_{n,k} \otimes_{\min} C^*(\bb F_{\infty})= 
 \cl {NP}_{n,k} \otimes_{\max} C^*(\bb F_{\infty})
$$
for some $n\geq 2,$ $k\geq 2$ with $(n,k) \neq (2,2)$ \cite{FKPT-discrete}. 

As a first step we shall prove that
a canonical order isomorphism between $C^*(\bb F_{\infty}) \otimes_{\min} \cl {NP}_{n,k}$ and
$C^*(\bb F_{\infty}) \otimes_{\max} \cl {NP}_{n,k}$ implies a complete order isomorphism.
Let
$$
[x_{ij}] \geq 0 \mbox{ in } M_n( \cl {NP}_{n,k} \otimes_{\min} C^*(\bb F_{\infty})) \cong  \cl {NP}_{n,k} \otimes_{\min} M_n( C^*(\bb F_{\infty})) .
$$
We pick a finite dimensional operator subsystem $\cl S_0 $ of $M_n( C^*(\bb F_{\infty})) $ such that 
$[x_{ij}] $ belongs to $ \cl {NP}_{n,k} \otimes_{\min} \cl S_0 $.
Let $\pi: C^*(\bb F_{\infty}) \rightarrow M_n( C^*(\bb F_{\infty}))$ be a surjective $*$-homomorphism. 
Since $M_n( C^*(\bb F_{\infty}))$ has the local lifting property (see \cite{Kirchberg94} e.g.)
there exists a ucp map $\gamma: \cl S_0 \rightarrow  C^*(\bb F_{\infty})$ such that $\pi \circ \gamma = i$, 
where $i$ denotes the inclusion $\cl S_0 \subset M_n( C^*(\bb F_{\infty})) $. This means that
$ id \otimes \gamma  ([x_{ij}]) $ must be a positive element of $ \cl {NP}_{n,k} \otimes_{\min} C^*(\bb F_{\infty})$.
Since we assume that min and max coincide order ismorphically on $\cl {NP}_{n,k} \otimes C^*(\bb F_{\infty})$, it follows that 
$id\otimes \gamma ([x_{ij}]) \geq 0 $ in $  \cl {NP}_{n,k} \otimes_{\max} C^*(\bb F_{\infty})$.
By using the functoriality of the maximal tensor product we deduce that
$id \otimes \pi (  id  \otimes \gamma ([x_{ij}])  ) = [x_{ij}]$ must be positive in $\cl {NP}_{n,k} \otimes_{\max} M_n( C^*(\bb F_{\infty}))$.
This proves our claim.

(1) $\Leftrightarrow$ (2). Letting $Q: (\ell^{\infty}_{nk})^* \rightarrow \cl W_{n,k}^* \cong \cl {NP}_{n,k}$ be the 
adjoint of the inclusion $\cl W_{n,k} \subset \ell^{\infty}_{n,k}$, we can rephrase the above formulation
of the Connes' embedding problem as follows:
$$
(\ell^{\infty}_{nk})^* \otimes_{} C^*(\bb F_{\infty}) \xrightarrow{\;\; Q \otimes id\;\;}  
\cl W_{n,k}^* \otimes_{\min} C^*(\bb F_{\infty})\cong \cl {NP}_{n,k} \otimes_{\min} C^*(\bb F_{\infty})
$$
is an order quotient map. Indeed, by the projectivity of the maximal tensor product,
$$
(\ell^{\infty}_{nk})^* \otimes_{} C^*(\bb F_{\infty}) \xrightarrow{\;\; Q \otimes id\;\;}  
\cl W_{n,k}^* \otimes_{\max} C^*(\bb F_{\infty})
$$
is readily an order quotient map, thus we obtain an order isomorphism between
$\cl W_{n,k}^* \otimes_{\min} C^*(\bb F_{\infty})$ and $\cl W_{n,k}^* \otimes_{\max} C^*(\bb F_{\infty})$.
 By Lemma \ref{lem-CEP2}, this is equivalent to the statement that
every ucp map $\varphi: \cl W_{n,k} \rightarrow C^*(\bb F_{\infty})$ has an approximate
ucp extension on $\ell^{\infty}_{nk}$. This proves the equivalence of (1) and (2).

(1) $\Leftrightarrow$ (3). $ \cl {NP}_{n,k} \otimes_{\min} C^*(\bb F_{\infty})$ is order isomorphic to
 $ \cl {NP}_{n,k} \otimes_{\max} C^*(\bb F_{\infty})$ if and only if every positive linear functional
on $\cl {NP}_{n,k} \otimes_{\max} C^*(\bb F_{\infty})$ is also positive with respect to min. Based on the
representation of the maximal tensor product in \cite{kavruk2011} and Lemma \ref{lem-CEP1}, this is equivalent to
the statement that every completely positive map from $ C^*(\bb F_{\infty})$ to $(\cl {NP}_{n,k})^*$
is nuclear. Finally using the identification $(\cl {NP}_{n,k})^* \cong \cl W_{n,k}$ and noticing that every positive map
into $\cl W_{n,k,}$ is completely positive, we obtain the equivalence of (1) and (3).

\end{proof}

\section{Partition of Unity Property}

We have defined $m$-partition of unity property at the introduction. 
It is elementary to show that a unital C*-algebra $\cl A$ has $m$-partition of unity property
if and only if it satisfies the following: for every elements $b_1,...,b_m$ in $\cl A$ with
$$
-1 < b_1,...,b_m < 1
$$
there is a positive integer $n$, a positive element $[a_{ij}]$ in $M_n(\cl A)$ with $a_{11}+a_{22}+\cdots + a_{nn} = 1$,
$n\times n$ matrices $C_k = [c^k_{ij}]$ with $-I_n < C_k < I$ for $k=1,2,...,m$ such that
$$
b_k = \sum_{i,j=1}^n c^k_{i,j} a_{ij} \mbox{ for all } k=1,2,...,m.
$$
We leave the verification to the reader and in the rest we shall use this version of partition of unity property.

\begin{proposition}
A unital C*-algebra $\cl A$ has m-partition of unity property if and only if $  \cl W_{2,m} \otimes_{\max} \cl A \subset  \ell_{2m}^\infty \otimes \cl A$
order isomorphically. 
\end{proposition}
\begin{proof}
We first fix the basis
$$
1=w_0 = (1,1,...,1), \;\; w_1 = (1,-1,0,...,0), \;\; w_2 = (0,0,1,-1,0,...,0),\;\; w_m = (0,...,0,1,-1)
$$
for $\cl W_{2,m}$. First suppose that $\cl A$ has $m$-partition of unity property.
In order to prove that the ucp map $i \otimes id:  \cl W_{2,m} \otimes_{\max} \cl A \rightarrow  \ell_{2m}^\infty \otimes \cl A$
is an order embedding we must show the following:
if $\Sigma w_k \otimes a_k$ is positive in $ \ell_{2m}^\infty \otimes \cl A$ then it must be positive in $ \cl W_{2,m} \otimes_{\max} \cl A$.
A moment of thought shows that we may suppose $\Sigma w_k \otimes a_k$ is a strictly positive element in $\ell_{2m}^\infty \otimes \cl A$.
But this implies that $a_0$ must be strictly positive in $\cl A$. Since the functor $\cdot \, \otimes_{\max} \cl A$
yields an operator $\cl A$-system, by multiplying $\Sigma w_k \otimes a_k$ by $a_0^{-1/2}$ from both side we may suppose that
$a_0=1$. It is elementary to show that
$$
\sum_{k=0}^m w_k \otimes a_k > 0\;\; \mbox{ in } \ell_{2m}^\infty \otimes \cl A \;\;\;\; 
\mbox{ if and only if }\;\;\; -1< a_k < 1 \;\mbox{ for } k=1,...,m \;\mbox{ in }\; \cl A.
$$
So, let $n$, $[a_{ij}]$, $C_1,...,C_m$ be as in the above formulation of the $m$-partition of unity property. We set the vector
$x = e_1\otimes e_1 + \cdots + e_n \otimes e_n$ in $\mathbb{C}^n \otimes \mathbb{C}^n$, where $e_i$ is the vector whose $i^{th}$-entry is 1 and 0 elsewhere.
Observe that 
$
x^* ( C_k \otimes [a_{ij}] ) x = \sum c_{ij}^k a_{ij} =  a_k.
$
So we must have $x^* ( C_k \otimes w_k \otimes  [a_{ij}] ) x = w_k \otimes a_k$, for $k = 1,...,m$.
Moreover, 
$$
x^* (I_n \otimes w_0 \otimes [a_{ij}]) \, x= w_0 \otimes (a_{11} + a_{22} + \cdots + a_{nn}) = w_0 \otimes 1.
$$
where $I_n$ is the identity in $M_n$. Thus,
\begin{eqnarray*}
  \sum_{k=0}^m w_k \otimes a_k &=&  x^* (I_n \otimes w_0 \otimes [a_{ij}]) \, x + \sum_{k = 1}^mx^* ( C_k \otimes w_k \otimes  [a_{ij}] ) x \\
 &=& x^* \left( I_n \otimes w_0 + C_1 \otimes w_1 + C_2 \otimes w_2 + \cdots  C_m \otimes w_m \right) \otimes [a_{ij}] \, x.
\end{eqnarray*}
Consequently, all we need to show is that $W = I_n \otimes w_0 + C_1 \otimes w_1 + C_2 \otimes w_2 + \cdots + C_m \otimes w_m$ is positive in $M_n \otimes \cl W_{2,m}$.
As $M_n \otimes \cl W_{2,m}$ embeds in $M_n \otimes (\ell_{2m}^\infty)$, completely order isomorphically, it is sufficient to show that
$W$ is positive in $M_n \otimes (\ell_{2m}^\infty)$. But this is a direct consequence of the fact that $-I_n < C_k < I_n$ for every $k$, in fact,
$$
W = (I_n + C_1, I_n - C_1,...,I_n + C_m, I_n - C_m).
$$
Now the positivity of $\Sigma w_k \otimes a_k$ in $\cl W_{2,m} \otimes_{\max} \cl A$ follows from the definition of the construction of the maximal cone.

Conversely, suppose that the inclusion $ \cl W_{2,m} \otimes_{\max} \cl A \hookrightarrow \ell_{2m}^\infty \otimes \cl A$ is an order isomorphism.
Let $-1 < b_1,...,b_m < 1$ be given.
Since
$$
(1+b_1, 1-b_1,...,1+b_m, 1-b_m ) = \sum_{k = 0}^m w_k \otimes b_k
$$
where $b_0= 1$, is a strictly positive element of $\ell_{2m}^\infty \otimes \cl A$, it is strictly positive in $\cl W_{2,m} \otimes_{\max} \cl A$.
By the definition of maximal tensor product, there exist $W \in M_n(\cl W_{n,k})^+$, $A \in M_p(\cl A)^+$ and a vector $x \in \bb C^n \otimes \bb C^p $,
for some integers $n$ and $p$, such that
$$
\sum_{k = 0}^m w_k \otimes b_k = x^* (W \otimes A) x.
$$
We can write $W = \sum C_{i} \otimes w_i$, where $C_0,...,C_m$ are matrices in $M_n$. Since $M_n \otimes \cl W_{n,k} \subset M_n \otimes \ell_{nk}^{\infty}$
completely order isomorphically, we must have that
$$
W = \sum C_{i} \otimes w_i = (C_0 + C_1, C_0 - C_1,...,C_0 + C_m, C_0 - C_m) \geq 0
$$
in $M_n \otimes \ell_{nk}^{\infty}$. Note that this implies $C_0 \geq 0$ with $-C_0 \leq C_i \leq C_0$ for $i=1,...,m$. Since 
$$
\sum_{k = 0}^m w_k \otimes b_k = x^* (\left( \sum_{k = 0}^m C_{i}  \otimes w_i \right) \otimes A) x = \sum_{k = 0}^m w_k \otimes (x^* (C_i \otimes A) x)
$$
we obtain that $b_k = x^* (C_k \otimes A) x$, $k=0,1,...,m$ with $b_0 = 1_{\cl A}$. We claim that, by replacing $x$ if necessary,
we may suppose $C_0$ is the identity matrix. In fact, let $P$ be the support projection of $C_0$, and hence $C_0^{1/2}$ in $M_n$. 
Let $D$ be a positive partial inverse of $C_0^{1/2}$
(so that $ C_0^{1/2} D = P$). Define $y = (C_0^{1/2} \otimes I) x$. First note that
$$
1_{\cl A} = x^* (C_0 \otimes A) x = x^* (C_0^{1/2} \otimes I) (I\otimes A) (C_0^{1/2} \otimes I) x = y^* (I \otimes A) y.
$$
Now $ -C_0 \leq C_k \leq C_0$ implies that $P C_k P = C_k$. Thus, setting $\tilde C_k = D C_k D$ we have $-I \leq \tilde C_k \leq I$
and
$$
b_k = x^* (C_k \otimes A) x = x^* (C_0^{1/2} \otimes I) (\tilde C_k \otimes A) (C_0^{1/2} \otimes I) x  = y^* (\tilde C_k \otimes A) y.
$$
We can write $y = e_1 \otimes x_1 + \cdots  + e_n \otimes x_n  \in \bb C^n \otimes \bb C^p$. We define $a_{ij} = x_i^*A x_j$ for $1\leq i,j \leq n$. 
Note that
$[a_{ij}]$ is a positive element in $M_n(\cl A)$ as
$$
[x_1 \; x_ 2 \; \cdots x_n]^* A  [x_1 \; x_ 2 \; \cdots x_n] = [a_{ij}].
$$
Moreover, 
$$
a_{11} + a_{22} + \cdots + a_{nn} = \sum_{i=1}^n x_i^* A x_i = y^* (I \otimes A) y = 1_{\cl A}.
$$
Finally,
$$
b_k = y^* (\tilde C_k \otimes A) y = (\Sigma e_i \otimes x_i)^*  (\tilde C_k \otimes A) (\Sigma e_i \otimes x_i) = \sum_{i,j=1}^n c_{ij}^k a_{ij}
$$
where $c_{ij}^k$ is the $(i,j)^{th}$ entry of $\tilde C_k$.

In order to finish the proof, we need to show that the inequalities $-I_n \leq   \tilde C_k \leq I_n$ can be choosen to be strict.
In fact the condition $-1 < b_1,...,b_m < 1$ guarantees that we can find a small $\epsilon >0$
so that $$-1 < \frac{b_i}{1-\epsilon} < 1 ; \;\; i = 1,...,m.$$ If we repeat the above machinary with $b_i / (1-\epsilon)$ is replaced by $b_i$,
we can find a positive element  [$a_{ij}$] with $\Sigma a_{ii} = 1_{\cl A}$, and matrices $\tilde C_i = [c_{ij}^k]$ with $-I_n \leq C_k \leq I_n$
with $\Sigma c_{ij}^k a_{ij}= b_i / (1-\epsilon)$. Then it is straightforward to see that the matrices $(1-\epsilon) C_{k}$, which staisfies
$ -I_n < (1-\epsilon) C_{k} < I_n$, along with the positive matrix $[a_{ij}]$ can be used for the partitioning of $b_1,...,b_m$. This completes the proof.  \end{proof}

\begin{proof}[proof of Theorem \ref{POUP}] (1) By the above proposition, a unital C*-algebra has 2-partition of unity property if and only if $\cl A \otimes_{\max} \cl W_{2,2} \subset \cl A \otimes \ell_{4}^\infty$. But this is a direct consequence of Theorem \ref{thm-W22}.

(2) By the above proposition, $M_n(\cl A)$ has $m$-partition of
unity property for every $n$ if and only if $M_n(\cl A) \otimes_{\max} \cl W_{2,m} \subset M_n(\cl A) \otimes \ell_{2m}^\infty$. This is equivalent to the statement that
$$
\cl A \otimes_{\max} \cl W_{2,m} \subset \cl A \otimes \ell_{2m}^\infty
$$
completely order isomorphically. Since the operator system structure on $\cl A \otimes \cl W_{2,m} $ arising from the inclusion $\cl A \otimes \ell_{2m}^\infty$ coincide with
$\cl A \otimes_{\min} \cl W_{2,m} $, we conclude that
$M_n(\cl A)$ has $m$-partition of unity property for every $n$ if and only if
$$
\cl A \otimes_{\min} \cl W_{2,m} = \cl A \otimes_{\max} \cl W_{2,m}
$$
completely order isomorphically. But the later condition, for some $m \geq 3$ or for every $m$, is equivalent to $\cl A$ being nuclear. This finishes the proof.
\end{proof}

\section{Appendix}

In this section we prove that, for finite dimensional operator systems, C*-nuclearity is preserved under duality. 
Let $\cl S$ be a finite dimensional operator system with basis $\{s_1,...,s_n\}$ and let
$\{\delta_1,...,\delta_n\}$ be the corresponding dual basis (i.e. $\delta_i(s_j) = \delta_{ij}$).
Then we define the {\it maximally entangled element} by
$$
s_1 \otimes \delta_1 + \cdots + s_n \otimes \delta_n \in \cl S \otimes \cl S^*.
$$
This element is canonical in the sense that it is independent of a particular choice of a basis. Moreover,
$$
s_1 \otimes \delta_1 + \cdots + s_n \otimes \delta_n \geq 0 \mbox{ in } \cl S \otimes_{\min} \cl S^*.
$$
Indeed, by the representation of the minimal tensor product (i.e. $CP(\cl S,\cl S) \cong (\cl S \otimes_{\min} \cl S^*)^+$),
it is easy to see that this canonical element corresponds to the identity.

Before we proceed, recall that the {\it commuting tensor product} of two operator system, $\cl S \otimes_{\rm c} \cl T$, is defined via ucp maps with commuting ranges. More precisely, we define the cone of positive elements by
\begin{eqnarray*}
C^{\rm com}_n = \{ X \in M_n(\cl S \otimes \cl T) :& (\phi\otimes\psi)_n (X) \geq 0 \; \forall\mbox{ ucp maps }\phi: \cl S \rightarrow B(H),\, \psi: \cl T \rightarrow B(H)\\
& \mbox{with commuting ranges and for all Hilbert space } H \}.
\end{eqnarray*}
The commuting tensor product is associative. An operator system $\cl S$ is called (min,c)-nuclear if $\cl S \otimes_{min} \cl T = \cl S \otimes_{\rm c} \cl T$ for all $\cl T$. As pointed out in \cite{kavruk2011}, (min,c)-nuclerity is equivalent to C*-nuclerity.

\begin{theorem}
The following are equivalent for a finite dimensional operator system $\cl S$:
\begin{enumerate}
 \item $\cl S$ is C*-nuclear;
 \item maximally entangled element $s_1 \otimes \delta_1 + \cdots + s_n \otimes \delta_n \geq 0 $ is positive in $ \cl S \otimes_{\rm c} \cl S^*$;
 \item $\cl S^*$ is C*-nuclear.
\end{enumerate}
\end{theorem}

\begin{proof}
Once we prove that (1) and (2) are equivalent then it is easy to observe that (3) is equivalent these two condition. 
In fact this is a direct consequence of the fact that the commuting tensor product is associative.
Moreover, since the canonical element is positive
in the minimal tensor product it is elementary to see that (1) implies (2) (recall that C*-nuclearity is equivalent to (min,c)-nuclearity).
Therefore we shall prove (2)$\Rightarrow$(1):

As a first step we wish to prove that the canonical embedding $\cl S \hookrightarrow C^*_u(\cl S)$ is nuclear.
We assume that each element of the basis $\{s_1,...,s_n\}$ is contractive. 
Since $ \cl S \otimes_{\rm c} \cl S^* \subset C^*_u(\cl S) \otimes_{\max} \cl S^*$ completely order isomorphically
it follows that
$$
s_1 \otimes \delta_1 + \cdots + s_n \otimes \delta_n \geq 0 \mbox{ in } C^*_u(\cl S) \otimes_{\max} \cl S^*.
$$
This means that for every $\epsilon > 0$ there are integers $k, m$, positive elements $(T_{ij})$ in $M_k(C^*_u(\cl S))$
and $(f_{ij})$ in $M_{m}(\cl S^*)$ and $x \in \bb C^{mk}$ (all depending on $\epsilon$) such that
$$
x^* \; (T_{ij}) \otimes (f_{ij})\; x = \Sigma s_i \otimes \delta_i + \epsilon (e \otimes \delta)
$$
where $\delta$ denotes a faithful state on $\cl S$ fixed as an Archimedean matrix order unit for $\cl S^*$.
Define
$$
\phi_{\epsilon} : \cl S \rightarrow M_m \mbox{ by } s \mapsto (f_{ij}(s)).
$$
The definition of dual matricial cones ensures that $\phi_{\epsilon}$ is completely positive.
Similarly define
$$
\psi_{\epsilon}: M_{m} \rightarrow C^*_u(\cl S) \mbox{ by } A \mapsto x^* \; (T_{ij}) \otimes A \; x.
$$
Complete positivity of $\psi_{\epsilon}$ follows from the compatibility condition of matricial
cone structure of an operator system and we skip this routine procedure. Now we fix one of the basis elements $s_{l}$.
Note that
$$
\psi_\epsilon \circ \phi_\epsilon (s_l) = x^* \; (T_{ij}) \otimes (f_{ij}(s_l))\; x = 
\Sigma s_i \delta_i (s_l) + \epsilon (e  \delta (s_l)) = s_l + \epsilon \delta(s_l)e.
$$
Consequently we deduce that
$$
\| \psi_\epsilon \circ \phi_\epsilon (s_l)  - s_l  \| = \|\epsilon \delta(s_l)e\| \leq \epsilon.
$$
Since $\{s_i\}$ is a basis it is elementary to see that for any element $s \in \cl S$
$$
\| \psi_\epsilon \circ \phi_\epsilon (s)  - s  \|  \leq C \epsilon
$$
where  $C$ solely depends on the basis $\{s_1,...,s_n\}$ and independent of $\epsilon$.
Finally by a selection of a sequence $\epsilon_n \rightarrow 0$ it is easy to see that the canonical embedding
$\cl S \hookrightarrow C^*_u(\cl S)$ is nuclear.

Now we prove that $\cl S$ is C*-nuclear. Let $\phi_i: \cl S \rightarrow M_{n(i)}$ and $\psi_i:  M_{n(i)} \rightarrow C^*_u(\cl S)$
be the cp maps that converges to the canonical embedding of $\cl S$ into $C^*_u(\cl S)$. We fix a unital C*-algebra $\cl A$.
Then
$$
\cl S \otimes_{\min} \cl A \xrightarrow{\phi_i \otimes id } M_{n(i)} \otimes_{\min} \cl A = M_{n(i)} \otimes_{\max}
\cl A \xrightarrow{\psi_i \otimes id } C^*_u(\cl S)  \otimes_{\max} \cl A
$$
is a net of cp maps converges to $i \otimes id$ in point-norm topology where $i$ denotes the inclusion of $\cl S$ into $C^*_u(\cl S)$. 
Since point-norm limit of cp maps are again cp it follows that the canonical map
$$
\cl S \otimes_{\min} \cl A  \longrightarrow C^*_u(\cl S)  \otimes_{\max} \cl A
$$
is cp. Since $\cl S \otimes_{\max} \cl A \subset C^*_u(\cl S)  \otimes_{\max} \cl A$ completely order isomorphically
it follows that the identity map
$$
\cl S \otimes_{\min} \cl A  \longrightarrow \cl S \otimes_{\max} \cl A
$$
is completely positive. This finishes the proof.
\end{proof}

$ $

{\sc Department of Mathematics, University of Illinois Urbana-Champaign

Urbana, IL 61801, U.S.A.}

{\it E-mail address:} kavruk@illinois.edu

\end{document}